\documentclass[12pt, oneside]{amsart}

\usepackage{amsmath,amsfonts,amssymb,amsthm}
\usepackage[english]{babel}

\newtheorem{theorem}{Theorem}[section]

\newtheorem{lemma}{Lemma}[section]
\newtheorem{question}{Question}[section]

\newtheorem{example}{Example}[section]

\newtheorem{remark}[theorem]{Remark}

\begin{document}

\title{Caratheodory's solution of the Cauchy problem and question Z.Grande}

\author{Volodymyr Mykhaylyuk, Vadym Myronyk}
\address{ Chernivtsi National University,
                    Department of Mathematical Analysis, Department of Algebra and Computer Science, Ukraine}
\date{}



\begin{abstract}
It is shown that for a function $f:\mathbb R^2\to \mathbb R$ which is measurable with respect to the first variable and upper semicontinuous quasicontinuous and increasing with respect to the second variable there exists a Caratheodory's solution $y(x)=y_0+\int\limits_{x_0}^xf(t,y(t))d\mu(t)$ of the Cauchy problem $y'(x)=f(x,y(x))$ with the initial condition $y(x_0)=y_0$. There are constructed examples which indicate to essentiality of condition of increasing and give the negative answer to a question of Z.~Grande.
\end{abstract}

\email{vmykhaylyuk@ukr.net, vadmyron@gmail.com}

\maketitle

\section{Introduction}

According to Scorza Dragoni's classical theorem \cite{SD}, for a function $f:[a,b]\times [c,d]\to\mathbb R$ which is measurable with respect to the first variable and continuous with respect to the second variable and every $\varepsilon>0$ there exists a set $A\subseteq [a,b]\times [c,d]$ such that $\mu(([a,b]\times [c,d])\setminus A)<\varepsilon$ and the restriction $f|_A$ is jointly continuous. In particular, $f$ is an $(L)-sup$-measurable, that is for every measurable function $\varphi:[a,b]\to [c,d]$ the function $f(x,\varphi(x))$ is measurable. This property plays an important role in the Caratheodory Differential Equations Theory (see \cite{F}) and was developed by many mathematicians (see \cite{T}, \cite{K} and the literature given there). The next result was obtained by Z.~Grande in recent paper \cite{G}.

\begin{theorem}\label{th:1.1}
Let a bounded fuction $f:\mathbb{R}^2\rightarrow\mathbb{R}$ satisfies the following conditions:
\begin{enumerate}
  \item[(1)] for every $x\in X$ the vertical section $f^x:\mathbb{R}\rightarrow\mathbb{R}$ is quasicontinuous and upper semicontinuous;
  \item[(2)] there exists a countable dense in $\mathbb{R}$ set $B$ such that for every $y\in B$ the horizontal section $f_y:\mathbb{R}\rightarrow\mathbb{R}$ is Lebesgue measurable.
\end{enumerate}
Then $f$ is an $(L)-sup$-measurable function.
\end{theorem}

In \cite{G} Z.~Grande posed the following question in the connection with Theorem \ref{th:1.1}.

\begin{question}[Z.~Grande, \cite{G}]\label{q:1.2}
Let a locally bounded function $f:\mathbb{R}^2\rightarrow\mathbb{R}$ satisfies conditions $(1)$ and $(2)$ of Theorem \ref{th:1.1}. Does there exist a Caratheodory's solution of the Caushy problem $y'(x)=f(x,y(x))$ with the initial condition $y(x_0)=y_0$?
\end{question}

In this paper we give a more general variant of Theorem \ref{th:1.1} and apply it to show that under more general assumptions than in Question \ref{q:1.2} there exists a unique maximal weak variant of Caratheodory's solution of the Caushy problem. This implies the positive answer to Question \ref{q:1.2}, if all vertical sections $f^x$ are increasing. Moreover, we construct two examples showing the essentiality of condition of increasing and give the negative answer to Question \ref{q:1.2}, even when $f$ depends on the second variable only and have the Darboux property.

\section{Measurability of Caratheodory superposition}

 For any function $f:X\times Y\to\mathbb R$ and for all $x\in X$ and $y\in Y$ we put $f^x(y)=f_y(x)=f(x,y)$.

The  following theorem is a more general variant of Theorem \ref{th:1.1}. It shows that the condition of boundness of $f$ in Theorem \ref{th:1.1} is not necessary.

 \begin{theorem}\label{t:3.1}
Let a function $f:\mathbb{R}^2\rightarrow\mathbb{R}$ satisfies the following conditions:
\begin{enumerate}
  \item[(1)] there exists a set $A\subseteq\mathbb{R}$ with $\mu(\mathbb{R}\setminus A)=0$ such that for every $x\in A$ the vertical section $f^x:\mathbb{R}\rightarrow\mathbb{R}$ is quasicontinuous and upper semicontinuous;
  \item[(2)] there exists a dense subset $B\subseteq \mathbb R$ such that for every $y\in B$ the horizontal section $f_y:\mathbb{R}\rightarrow\mathbb{R}$ is Lebesgue measurable.
\end{enumerate}
Then $f$ is $(L)-sup$-measurable.
\end{theorem}

\begin{proof}
For every $n\in \mathbb{N}$ we choose a disjoint countable family $(I_{n,k})_{k\in\mathbb{Z}}$ of  intervals $I_{n,k}=(a_{n,k},a_{n,k+1})$ such that
\begin{enumerate}
  \item [(a)] $a_{n,k}\in B;$
  \item [(b)] $a_{n,k+1}-a_{n,k}<\frac{1}{n};$
  \item [(c)] $\bigcup\limits_{k\in \mathbb{Z}}(a_{n,k},a_{n,k+1})\bigcup \{a_{n,k}:k\in\mathbb{Z}\}=\mathbb{R}.$

\end{enumerate}
Now we put
$$
f_n(x,y)=\left\{\begin{array}{lc}
  \min\{n,\sup\limits_{t\in I_{n,k}}f(x,t)\},\,y\in I_{n,k}, \\
  f(x,a_{n,k}),\, y=a_{n,k}. \\
\end{array}\right.
$$
and show that $\lim\limits_{n\rightarrow \infty} f_n(x,y)=f(x,y)$ for every $x\in A$ and $y\in \mathbb{R}$.

Fix $x_0\in A,\,y_0\in \mathbb{R},\,\varepsilon>0$. Since
$f^{x_0}$ is upper semicontinuous at $y_0$, there exists $\delta >0$
such that $f(x_0,y)\leq f(x_0,y_0)+\varepsilon$ for all $y\in
(y_0-\delta,y_0+\delta)$. Choose $n_0\in\mathbb{N}$ such that
$n_0\geq\max\{f(x_0,y_0)+\varepsilon,\frac{1}{\delta}\}$. Let
$n\geq n_0$. If $y_0\in\{a_{n,k}:k\in\mathbb{Z}\}$, then
$f_n(x_0,y_0)=f(x_0,y_0)$. We assume that $y_0\notin
\{a_{n,k}:k\in\mathbb{Z}\}$ and choose $m\in\mathbb{Z}$ such that
$y_0\in (a_{n,m},a_{n,m+1})$. According to $(b)$ and the choice of $n_0$, we have
$a_{n,m+1}-a_{n,m}< \frac{1}{n}\leq\delta$. Therefore, $I_{n,m}\subseteq
(y_0-\delta,y_0+\delta)$. This implies, in particular, that $f(x_0,y)\leq
f(x_0,y_0)+\varepsilon\leq n_0\leq n$ for every $y\in I_{n,m}$. Now we have
$$
f_n(x_0,y_0)=\min \{n,\sup\limits_{y\in
I_{n,m}}f(x_0,y)\}=\sup\limits_{y\in I_{n,m}}f(x_0,y)\leq
f(x_0,y_0)+\varepsilon.
$$
Using the inequality $f_n(x_0,y_0)\geq f(x_0,y_0)$ we obtain that
$$
|f_n(x_0,y_0)-f(x_0,y_0)|=f_n(x_0,y_0)-f(x_0,y_0)\leq\varepsilon
$$
for all $n\geq n_0$.

Let $g:\mathbb{R}\rightarrow\mathbb{R}$ be a measurable function. To obtain the measurability of $h(x)=f(x,g(x))$ it is sufficient to prove that all functions $h_n(x)=f_n(x,g(x))$ are measurable.

Fix an integer $n$ and put $A_{n,k}=\{x\in
\mathbb{R}:\, g(x)=a_{n,k}\}$, $B_{n,k}=\{x\in \mathbb{R}:\, g(x)\in
I_{n,k}\}$. Clearly, for every $k\in\mathbb{Z}$ the sets $A_{n,k}$ and $B_{n,k}$ are measurable and
$\mathbb{R}=\bigsqcup\limits_{k\in\mathbb{Z}}(A_{n,k}\sqcup B_{n,k})$. Therefore, it is enough to prove that the restrictions $h_n|_{A_{n.k}}$ and
$h_n|_{B_{n.k}}$ are measurable.

Since $h_n|_{A_{n.k}}=f_{a_{n,k}}|_{A_{n.k}}$ and
$f_{a_{n,k}}$ is measurable, $h_n|_{A_{n.k}}$ is measurable.

We note that $h_n|_{B_{n.k}}=\gamma|_{B_{n.k}}$, where $\gamma
(x)=\min \{n,\sup\limits_{t\in I_{n,k}}f(x,t)\}$. We show that $\gamma$ is measurable. It follows from the quasicontinuity of
$f^x$ for $x\in A$, and the density of $B$ that
$$
\{x\in A:\gamma(x)>\alpha\}=\bigcup\limits_{b\in B\cap
I_{n,k}}\{x\in A:f(x,b)>\alpha\}.
$$
Therefore, condition $(2)$ implies that $\{x\in A:\gamma(x)>\alpha\}$
ia a measurable set for every $\alpha\in\mathbb{R}$. Thus, $\gamma |_A$ is a measurable function, hence $\gamma$ is measurable. Therefore, $h_n|_{B_{n.k}}$ is measurable.
\end{proof}

\section{Caratheodory's solution for functions with increasing vertical sections}

\begin{lemma}\label{t:4.2}
Let $f:[a,b]\times [c,d]\rightarrow\mathbb{R}$ be $(L)-sup$-measurable, $A\subseteq [a,b]$ be a set with $\mu ([a,b]\setminus A)=0$ such that for every $x\in A$ the vertical section $f^x$ is upper semicontinuous, let $(y_n)_{n=1}^\infty$ be a sequence of measurable functions $y_n:[a,b]\rightarrow [c,d]$ which converges to a measurable function $y_0:[a,b]\rightarrow [c,d]$ pointwisely on $[a,b]$. Then for the function $z:[a,b]\rightarrow [c,d],\,z(x)=\overline{\lim\limits_{n\rightarrow \infty}}f(x,y_n(x))$, we have
$$
\int\limits_a^b z(x)d\mu(x)\leq \int\limits_a^b f(x,y_0(x))d\mu(x).
$$
\end{lemma}
\begin{proof}
We show firstly that
$$
z(x)\leq f(x,y_0(x))
$$
for every $x\in A$.

Fix $x\in A$ and $\varepsilon >0$. Since $f^x$ is upper semicontinuous at $y_0(x)$, there exists $\delta >0$ such that
$$
f(x,y)<f(x,y_0(x))+\varepsilon
$$
for all $y\in [c,d]\cap (y_0(x)-\delta,y_0(x)+\delta)$. Using $\lim\limits_{n\rightarrow \infty}y_n(x)= y_0(x)$ we obtain that there exists $N\in\mathbb{N}$ such that
$y_n(x)\in (y_0(x)-\delta,y_0(x)+\delta)$
for all $n\geq N$. Thus,
$f(x,y_n(x))< f(x,y_0(x))+\varepsilon$
for all $n\geq N$. Therefore,
$\overline{\lim\limits_{n\rightarrow \infty}}f(x,y_n(x))\leq f(x,y_0(x))+\varepsilon$
for every $\varepsilon >0$. Hence,
$$z(x)=\overline{\lim\limits_{n\rightarrow \infty}}f(x,y_n(x))\leq f(x,y_0(x)).$$
Since $f$ is an $(L)-sup$-measurable function, the function $f(x,y_0(x))$ is measurable. It remains to integrate this inequality on $[a,b]$.
\end{proof}

\begin{theorem}\label{t:4.3}
Let a function $f:[a,b]\times\mathbb{R}\rightarrow\mathbb{R}$ satisfies the following conditions:
\begin{enumerate}
  \item [(1)] there exists a set $A\subseteq [a,b]$ with $\mu (A)=b-a$ such that for every $x\in A$ the vertical section $f^x:\mathbb{R}
  \rightarrow\mathbb{R}$ is quasicontinuous and upper semicontinuous;
  \item [(2)] there exists an everywhere dense set $B\subseteq\mathbb{R}$ such that for every $y\in B$ the horizontal section $f_y:[a,b]\rightarrow\mathbb{R}$ is measurable;
  \item [(3)] there exists an integrable function $\varphi :[a,b]\rightarrow\mathbb{R}$ such that $|f(x,y)|\leq\varphi(x)$ for every $(x,y)\in A\times\mathbb{R}$.
\end{enumerate}
Then for every $y_0\in\mathbb{R}$ there exists an unique maximal absolutely continuous function $z_0:[a,b]\rightarrow\mathbb{R}$ in the set of all absolutely continuous functions $z:[a,b]\rightarrow\mathbb{R}$ with $z(x_2)-z(x_1)\leq\int\limits_{x_1}^{x_2}f(t,z(t))d\mu(t)$ for arbitrary $a\leq x_1\leq x_2\leq b$  such that $z_0(a)=y_0$ and
$$
z_0(x_2)-z_0(x_1)\leq\int\limits_{x_1}^{x_2}f(t,z_0(t))d\mu(t)
$$
for every $a\leq x_1\leq x_2\leq b$.

If, moreover, $f$ satisfies the condition
\begin{enumerate}
  \item [(4)] for every $x\in A$ the vertical section $f^x:\mathbb{R}
  \rightarrow\mathbb{R}$ is increasing,
\end{enumerate}
then $z_0(x)=y_0+\int\limits_{a}^{x}f(t,z_0(t))d\mu(t)$.
\end{theorem}

\begin{proof}
We denote by $C_\varphi$ the set of all absolutely continuous functions $z:[a,b]\rightarrow\mathbb{R}$ with $|z(x_2)-z(x_1)|\leq\int\limits_{x_1}^{x_2}\varphi(t)d\mu(t)$ for every $a\leq x_1\leq x_2\leq b$. Let $\mathcal{F}$ be a set of all absolutely continuous functions $z\in C_\varphi$ such that $z(x_2)-z(x_1)\leq\int\limits_{x_1}^{x_2}f(t,z(t))d\mu(t)$ for every $a\leq x_1\leq x_2\leq b$ and $z(a)=y_0$.

We show that $\mathcal{F}\neq\emptyset$. Consider the function $w:[a,b]\rightarrow\mathbb{R},\,w(x)=y_0-\int\limits_a^x\varphi (t)d\mu(t)$. Clearly, $w$ is an absolutely continuous function, because $\varphi$ is integrable. Moreover, $w(a)=y_0$ and
$$
w(x_2)-w(x_1)=-\int\limits_{x_1}^{x_2}\varphi (t)d\mu(t)\leq -\int\limits_{x_1}^{x_2}|f(t,w(t))|d\mu(t)\leq \int\limits_{x_1}^{x_2}f(t,w(t))d\mu(t)
$$
for every $a\leq x_1\leq x_2\leq b$.

Now we show that $\mathcal{F}$ has a maximal element.

Let $\mathcal{A}\subseteq \mathcal{F}$ be a linearly ordered set. We show that $\mathcal{A}$ is upper bounded in $\mathcal{F}$. If $\mathcal{A}$ has a maximal element, then the statement is clear. Suppose that $\mathcal{A}$ have not any maximal element.

Note that
$$
z(x)\leq y_0+\int\limits_a^x\varphi (t)d\mu(t)=w_0(x)
$$
for every $z\in\mathcal{F}$ and every $x\in [a,b]$. Therefore, in particular, $z(x)\leq w_0(x)$ for every $z\in \mathcal{A}$ and every $x\in [a,b]$.

We put $z_1(x)=\sup\limits_{z\in\mathcal{A}}z(x)$. Clearly, $z_1(x)\leq w_0(x)$ for every $x\in[a,b]$. We show that there exists an increasing sequence $(w_n)_{n=1}^\infty$ of function $w_n\in \mathcal{A}$ such that $z_1(x)=\lim\limits_{n\to\infty}w_n(x)$ for every $x\in [a,b]$. We consider the set
$$
I=\left\{\int\limits_a^bw(x)d\mu(x):w\in \mathcal{A}\right\}
$$
and choose an increasing sequence $(\alpha_n)_{n=1}^\infty$ of reals $\alpha_n\in I$ such that $\sup I=\lim\limits_{n\to\infty}\alpha_n$. Now for each $n\in\mathbb N$ we choose $w_n\in \mathcal{A}$ such that $\alpha_n=\int\limits_a^bw_n(x)d\mu(x)$. The linear ordering of $\mathcal{A}$ implies that $(w_n)_{n=1}^\infty$ increases. Suppose that $z_1\ne\lim\limits_{n\to\infty} w_n$. Then there exists $w\in\mathcal{A}$ such that $w\geq w_n$ for every $n\in\mathbb N$. Then $\int\limits_a^bw(x)d\mu(x)=\sup I$. Since $\mathcal{A}$ is a linearly ordered set, $w$ is a maximal element in $\mathcal{A}$, which is impossible. Thus, $z_1=\lim\limits_{n\to\infty} w_n$.

We show that $z_1\in \mathcal{F}$. Fix $a\leq x_1< x_2\leq b$. By Lemma \ref{t:4.2} we have
$$
z_1(x_2)-z_1(x_1)=\lim\limits_{n\to\infty}(w_n(x_2)-w_n(x_1))\leq \overline{\lim\limits_{n\to\infty}}\int\limits_{x_1}^{x_2}f(t,w_n(t))d\mu(t)\leq \int\limits_{x_1}^{x_2}f(t,z_1(t))d\mu(t).
$$
Moreover,
$$
|z_1(x_2)-z_1(x_1)| =\lim\limits_{n\to\infty} |w_n(x_2)-w_n(x_1)| \leq\int\limits_{x_1}^{x_2}\varphi(t)d\mu(t).
$$

Thus, every linearly ordered set $\mathcal{A}\subseteq \mathcal{F}$ is upper bounded. Then, according to Kuratowski-Zorn Lemma, $\mathcal{F}$ has a maximal element $z_0$.

We prove that $z_0$ is a unique. Let $z_0$ and $u_0$ be maximal elements in $\mathcal{F}$. Suppose that $z_0\neq u_0$. For definiteness let $a< x_1< x_2\leq b$ be such that $z_0(x_1)> u_0(x_1)$ and $u_0(x_2)>z_0(x_2)$. Taking into account that $u_0(a)=z_0(a)=y_0$ we found a maximal $x_3\in [a.x_1)$ such that $z_0(x_3)=u_0(x_3)$. We choose a maximal $x_4\in (x_1,x_2)$ such that $z_0(x_4)=u_0(x_4)$. Then $z_0(x)>u_0(x)$ for all $x\in (x_3,x_4)$. We consider the function
$$
v(x)=\left\{\begin{array}{lc}
  u_0(x), & x\in [a,x_3]\cup [x_4,b], \\
  z_0(x), & x\in(x_3,x_4). \\
\end{array}\right.
$$
Since $u_0,z_0\in\mathcal{F}$, $v\in\mathcal{F}$, which contradicts to the maximality of $u_0$.

Now we prove that $z_0(x)=y_0+\int\limits_{a}^{x}f(t,z_0(t))d\mu(t)$ if $f$ satisfies condition $(4)$. We put
$$
v_0(x)=y_0+\int\limits_{a}^{x}f(t,z_0(t))d\mu(t).
$$
It follows from $(3)$ that $v_0\in C_\varphi$. Moreover, since $z_0\in\mathcal{F}$,
$$
z_0(x)=y_0+z_0(x)-z_0(a)\leq y_0+\int\limits_{a}^{x}f(t,z_0(t))d\mu(t)=v_0(x)
$$
for every $x\in[a,b]$. Now for every $a\leq x_1< x_2\leq b$ condition $(4)$ implies that
$$
v_0(x_2)-v_0(x_1)=\int\limits_{x_1}^{x_2}f(t,z_0(t))d\mu(t)\leq \int\limits_{x_1}^{x_2}f(t,v_0(t))d\mu(t).
$$
Thus, $v_0\in\mathcal{F}$. Since $z_0$ is a maximal, $z_0=v_0$.
\end{proof}

\section{Examples}

The following example shows that Question \ref{q:1.2} has the negative answer and condition $(4)$ in Theorem \ref{t:4.3} is essential.

\begin{theorem}\label{t:2.1}
There exists a quasicontinuous upper semicontinuous function $f:\mathbb{R}\rightarrow [-1,1]$ such that the Cauchy problem
$$
\left\{\begin{array}{lc}
  y'(x)=f(y(x)), \\
  y(0)=0, \\
\end{array}\right.\eqno (1)
$$
has not any Caratheodory's solution.
\end{theorem}
\begin{proof}

We consider the function $f(y)=\left\{\begin{array}{lc}
  -1, & y> 0, \\
  1, & y\leq 0. \\
\end{array}\right.$
Clearly, $f$ is a quasicontinuous upper semicontinuous function.

Suppose that there exists a Caratheodory's solution $y_0$ of the Cauchy problem (1), i.e.
$$
y_0(x)=\int\limits_0^x f(y(t))dt.\eqno (2)
$$
Consider the set $F=\{x\in \mathbb{R}:y_0(x)=0\}$. Since $y_0$ is absolutely continuous, $F$ is closed. We show that $F\neq\mathbb{R}$.

Suppose that $[0,1]\subseteq F$. Then $0=y_0(1)=\int\limits_0^1 f(y(t))dt=\int\limits_0^1 dt=1$, a contradiction.

Thus $F\neq \mathbb{R}$, that is, $G=\mathbb{R}\setminus F\neq \emptyset$. Let $G=\bigsqcup\limits_{s\in S}I_s$, where $I_s=(a_s,b_s)$. Fix $s_0\in S$ and put $I_0=(a_0,b_0)$. For difiniteness we assume that $a_0\in \mathbb{R}$. Then $a_0\in F$, that is, $y_0(a_0)=0$. Moreover, $y_0(x)\neq 0$ for every $x\in I_0$. Taking into account that $y_0$ is a continuous function we obtain that either $y(x)> 0$ for all $x\in I_0$, or $y(x)< 0$ for all $x\in I_0$.

Let $y(x)> 0$ for all $x\in I_0$. For any $x_0\in I_0$ we have
$$
0<y_0(x_0)=y_0(x_0)-y_0(a_0)=\int\limits_{a_0}^{x_0} y'(x)dx=-\int\limits_{a_0}^{x_0}dx=-(x_0-a_0)<0,
$$
a contradiction. The case of $y(x)< 0$ for all $x\in I_0$ is similar.
\end{proof}

\begin{remark}
For the given function $f$ the Cauchy problem
$$
\left\{\begin{array}{lc}
  y'(x)=f(y(x)), \\
  y(x_0)=y_0, \\
\end{array}\right.
$$
has not any Caratheodory's solution $y:\mathbb R\to\mathbb R$ for every initial condition.
\end{remark}

A function $f:[a,b]\to\mathbb R$ is called  a {\it Darboux function}, if for every interval $X\subseteq [a,b]$ the set $f(X)$ is an interval.

The following example shows that Question \ref{q:1.2} has the negative answer even when $f$ is a Darboux function with respect to the second variable.

\begin{example}
There exists an upper semicontinuous quasicontinuous Darboux function $f:\mathbb{R}\rightarrow[-1,1]$ such that the Cauchy problem $(1)$
has not any Caratheodory's solution.
\end{example}
\begin{proof}
We consider the function
$$
\left\{\begin{array}{lc}
  1, y\leq 0\\
  \sin \frac{\pi}{y}, y>0\,.\\
\end{array}\right.
$$
Clearly, $f$ is a upper semicontinuous quasicontinuous Darboux function.

Suppose that there exists an absolutely continuous function $y_0:\mathbb{R}\rightarrow\mathbb{R}$ which is a Caratheodory's solution of problem (1), i.e. $y_0$ has the form (2). We show that there exists $x_0>0$ such that $y_0(x_0)>0$. Assume that $y_0(x)\leq 0$ for all $x>0$. Then according to (2) for all $x>0$ we have
$$
y(x)=\int\limits_0^x f(y(t))dt=x,
$$
a contradiction.

We choose $x_0>0$ such that $y_0(x_0)>0$ and find $n_0\in \mathbb{N}$ such that $y_0(x_0)>\frac{1}{2n_0-1}$. Since $y_0(0)<\frac{1}{2n_0}<y_0(x_0)$, there exists a maximal $a\in (0,x_0)$ such that $y_0(a)=\frac{1}{2n_0}$. Then $y_0(x)>\frac{1}{2n_0}$ for all $x\in (a,x_0]$. Since $y_0(a)<\frac{1}{2n_0-1}<y_0(x_0)$, there exists a minimal $b\in (a,x_0)$ such that $y_0(b)=\frac{1}{2n_0-1}$. Then $y_0(x)<\frac{1}{2n_0-1}$ for all $x\in (a,b)$.

Hence $y_0(x)\in (\frac{1}{2n_0},\frac{1}{2n_0-1})$, in particular, $f(y_0(x))<0$ for all $x\in (a,b)$. Now we have
$$
\tfrac{1}{2n_0-1}=y_0(b)=y_0(a)+\int\limits_a^b f(y_0(x))dx\leq y_0(a)=\tfrac{1}{2n_0},
$$
a contradiction.
\end{proof}

\end{document}